\documentclass[oneside,english]{amsart}
\usepackage[T1]{fontenc}
\usepackage[latin9]{inputenc}
\usepackage{color}
\usepackage{amssymb}

\makeatletter
\numberwithin{equation}{section} 
\numberwithin{figure}{section} 
  \theoremstyle{plain}
  \newtheorem*{thm*}{Theorem}
  \theoremstyle{plain}
  \newtheorem{thm}{Theorem}[section]
  \theoremstyle{definition}
  \newtheorem{defn}[thm]{Definition}
  \theoremstyle{plain}
  \newtheorem{lem}[thm]{Lemma}
  \theoremstyle{plain}
  \newtheorem{prop}[thm]{Proposition}
  \theoremstyle{remark}
  
  \theoremstyle{remark}
  \newtheorem*{acknowledgement*}{Acknowledgement}

\@addtoreset{equation}{section}

\usepackage{babel}
\makeatother

\begin{document}

\title[Helmholtz conditions and symmetries]
{Helmholtz conditions and symmetries for the time dependent case
of the inverse problem of the calculus of variations}

\author{Ioan Bucataru and Oana Constantinescu}

\date{\today}

\address{``Al.I.Cuza'' University of Iasi, Faculty of Mathematics, 700506 Iasi, Romania}
\email{bucataru@uaic.ro, oanacon@uaic.ro}
\urladdr{http://www.math.uaic.ro/\textasciitilde{}bucataru/}
\urladdr{http://www.math.uaic.ro/\textasciitilde{}oanacon/}

\begin{abstract}
We present a reformulation of the inverse problem of the calculus
of variations for time dependent systems of second order ordinary
differential equations using the Fr\"olicher-Nijenhuis theory on
the first jet bundle, $J^1\pi$. We prove that a system of time
dependent SODE, identified with a semispray $S$, is Lagrangian if
and only if a special class, $\Lambda^1_S(J^1\pi)$, of semi-basic
$1$-forms is not empty. We provide global Helmholtz conditions to
characterize the class $\Lambda^1_S(J^1\pi)$ of semi-basic
$1$-forms. Each such class contains the Poincar\'e-Cartan 1-form
of some Lagrangian function. We prove that if there exists a
semi-basic $1$-form in $\Lambda^1_S(J^1\pi)$, which is not a
Poincar\'e-Cartan 1-form, then it determines a dual symmetry and a
first integral of the given system of SODE.
\end{abstract}

\subjclass[2000]{58E30, 34A26, 70H03, 49N45}

\keywords{semi-basic forms, Poincar\'e-Cartan 1-form, Poincar\'e
Lemma, Helmholtz conditions, inverse problem, time dependent SODE,
dual symmetry, first integral}

\maketitle

\section{Introduction}

In this work we present a reformulation of the inverse problem for
time dependent systems of second order ordinary differential
equations in terms of semi-basic 1-forms. In this approach we
solely make use of the Fr\"olicher-Nijenhuis theory on
$\pi_{10}:\, J^{1}\pi\rightarrow M$, the first jet bundle of an
$(n+1)$-dimensional, real, smooth manifold $M$, which is fibred
over $\mathbb{R}$, $\pi:M\rightarrow\mathbb{R}$. We characterize
when the time dependent system of SODE
\begin{eqnarray}
\frac{d^{2}x^{i}}{dt^{2}}+2G^{i}\left(t,x,\frac{dx}{dt}\right)=0
\label{sode}
\end{eqnarray}
is equivalent to the system of Euler-Lagrange equations
\begin{eqnarray}
\frac{d}{dt}\left(\frac{\partial L}{\partial
y^{i}}\right)-\frac{\partial L}{\partial x^{i}}=0,\quad
y^{i}=\frac{dx^{i}}{dt}, \label{el_equations}
\end{eqnarray}
for some smooth Lagrangian function $L$ on $J^{1}\pi$, in terms of
a class $\Lambda^1_S(J^1\pi)$ of semi-basic $1$-forms on $J^1\pi$.
This work is a natural extension of the time independent case
studied in \cite{bucataru09}. However, the time dependent
framework has particular aspects that result into some differences
of this approach from the one studied in \cite{bucataru09}.
Lagrangian systems of time independent differential equations are
always conservative and this is not the case in the time dependent
framework. In this approach we use the formalism developed for the
inverse problem of the calculus of variations to search for
symmetries as well. In other words, this approach gives the
possibility of searching for dual symmetries for the time
dependent system \eqref{sode} of SODE in the considered class
$\Lambda^1_S(J^1\pi)$ of semi-basic $1$-forms.

Necessary and sufficient conditions under which the two systems
\eqref{sode} and \eqref{el_equations} can be identified using a
multiplier matrix are usually known as the \emph{Helmholtz
conditions}. This inverse problem was entirely solved only for the
case $n=1$ by Darboux, in 1894, and for $n=2$ by Douglas, in 1941
\cite{douglas41}. A geometric reformulation of Douglas approach,
using linear connections arising from a system of SODE and its
associated geometric structures, can be found in \cite{crampin94,
sarlet02}. The relation between the inverse problem of the
calculus of variations and the condition of self-adjointness for
the equations of variation of the system \eqref{sode} was studied
by Davis in 1929, \cite{davis29}. In 1935, Kosambi
\cite{kosambi35}, has obtained necessary and sufficient
conditions, that were called latter Helmholtz conditions, for the
equations of variations of the system \eqref{sode} to be
self-adjoint. For various approaches to derive the Helmholtz
conditions in both autonomous and nonautonomous case, we refer to
Santilli \cite{santilli78}, Crampin \cite{crampin81}, Henneaux
\cite{henneaux82}, Sarlet \cite{sarlet82}, Marmo et al.
\cite{marmo89}, Morandi et al. \cite{morandi90}, Anderson and
Thompson \cite{anderson92}, Krupkov\'a and Prince
\cite{krupkova08}. See also \cite{bucataru09} for a reformulation
of the Helmholtz conditions in terms of semi-basic $1$-forms in
the time independent case.

In this paper we study the inverse problem of the calculus of
variations when the time dependent system \eqref{sode} of SODE is
identified with a semispray $S$ on the first jet bundle $J^1\pi$.
We seek for a solution of the inverse problem of the calculus of
variations in terms of semi-basic $1$-forms on $J^1\pi$. We first
show, in Theorem \ref{LVF}, that a semispray $S$ is a Lagrangian
vector field if and only if there exists a class of semi-basic
1-forms $\theta$ on $J^{1}\pi$ such that their Lie derivatives
$\mathcal{L}_{S}\theta$ are closed $1$-forms. We denote this class
by $\Lambda^1_S(J^1\pi)$. These results reformulate, in terms of a
semi-basic $1$-form, the results expressed in terms of a $2$-form,
obtained in \cite{anderson92, balachandran80, crampin81,
crampin84a, marmo89}. In Proposition \ref{prop:LVF} we strengthen
the results of Theorem \ref{LVF} and prove that, if nonempty, the
set of semi-basic $1$-forms $\Lambda^1_S(J^1\pi)$ always contains
the Poincar\'e-Cartan 1-form of some Lagrangian function $L$.
Moreover, we show that any semi-basic $1$-form $\theta \in
\Lambda^1_S(J^1\pi)$, which is not the Poincar\'e-Cartan 1-form of
some Lagrangian function, determines a \emph{first integral} and a
\emph{dual symmetry} of the Lagrangian system.

In Theorems \ref{thm:LST1} and \ref{thm:LST2} we characterize the
semi-basic $1$-forms of the set $\Lambda^1_S(J^1\pi)$, depending
if they represent or not Poincar\'e-Cartan 1-forms. First, we pay
attention to $d_J$-closed, semi-basic $1$-forms $\theta$, where
$J$ is the vertical endomorphism. This class of semi-basic
$1$-forms $\theta$ coincides with the class of Poincar\'e-Cartan
$1$-forms corresponding to the Lagrangian function
$L=i_{S}\theta$, see Lemma \ref{LemPoincare2}. For this class we
prove, in Theorem \ref{thm:LST1}, that the inverse problem has a
solution if and only if the Poincar\'e-Cartan $1$-form is
$d_h$-closed, where $h$ is the horizontal projector induced by the
semispray.

In Theorem \ref{thm:LST2} we formulate a coordinate free version
of the Helmholtz conditions in terms of semi-basic $1$-forms on
the first jet bundle $J^1\pi$. If there exists a semi-basic
$1$-form $\theta$ that satisfies the Helmholtz conditions, then
$\theta$ is equivalent (modulo $dt$) with the Poincar\'e-Cartan
$1$-form of some Lagrangian function. Moreover, if such $\theta$
is not $d_J$-closed then $i_Sd\theta$ is a dual-symmetry and
induces a first integral for the semispray $S$. This theorem gives
a characterization for \emph{conservative}, Lagrangian, time
dependent vector fields. To derive the Helmholtz conditions in
Theorem \ref{thm:LST2} we make use of the Fr\"olicher-Nijenhuis
theory on $J^1\pi$ developed in Section \ref{FN-theory}, as well
as of the geometric objects induced by a semispray that are
presented in Section \ref{gobjects}.

Since semi-basic $1$-forms play a key role in our work, we pay a
special attention to this topic in Section \ref{semibasic}. In
this section we prove a Poincar\'e-type Lemma for the differential
operator $d_J$ restricted to the class of semi-basic $1$-forms on
$J^1\pi$. The Poincar\'e-type Lemma will be very useful to
characterize those semi-basic $1$ forms that are the
Poicar\'e-Cartan $1$-forms of some Lagrangian functions.

An important tool in this work is discussed in Section \ref{dcv},
it is a tensor derivation on $J^1\pi$ that is called the dynamical
covariant derivative induced by a semispray. This derivation has
its origins in the work of Kosambi \cite{kosambi35}, where it has
been introduced with the name of "bi-derivative". The notion of
dynamical covariant derivative was introduced by Cari\~nena and
Martinez in \cite{carinena91} as a covariant derivative along the
bundle projection $\pi_{10}$. See also \cite{cantrijn96,
crampin96, jerie02, mestdag01, sarlet95}. It can be defined also
as a derivation on the total space $TJ^1\pi$, see \cite{massa94,
morando95}.

\section{Preliminaries} \label{preliminaries}

\subsection{The first jet bundle $J^{1}\pi$}

For a geometric study of time dependent systems of SODE the most
suitable framework is the affine jet bundle
$(J^{1}\pi,\pi_{10},M)$, see \cite{crampin96, krupkova97, massa94,
saunders89}. We consider an $(n+1)$-dimensional, real and smooth
manifold $M$, which is fibred over $\mathbb{R}$,
$\pi:M\rightarrow\mathbb{R}$. The first jet bundle of $\pi$ is
denoted by $\pi_{10}: J^{1}\pi\to M$,
$\pi_{10}(j_{t}^{1}\gamma)=\gamma(t),$ for $\gamma$ a local
section of $\pi$ and $j_{t}^{1}\gamma$ the first jet of $\gamma$
at $t$. A local coordinate system $(t,x^{i})_{i\in\overline{1,n}}$
on $M$, where $t$ represents the global coordinate on $\mathbb{R}$
and $(x^{i})$ the $\pi$-fiber coordinates, induces a local
coordinate system on $J^{1}\pi,$ denoted by $(t,x^{i},y^{i}).$
Submersion $\pi_{10}$ induces a \emph{natural foliation} on
$J^1\pi$. Coordinates $(t, x^i)$ are transverse coordinates for
this foliation, while $(y^i)$ are coordinates for the leafs of the
foliation.

Throughout the paper we assume that all objects are
$C^{\infty}$-smooth where defined. The ring of smooth functions on
$J^{1}\pi$, the $C^{\infty}$ module of vector fields on $J^{1}\pi$
and the $C^{\infty}$ module of $k$-forms are respectively denoted
by $C^{\infty}(J^{1}\pi)$, ${\mathfrak X}(J^{1}\pi)$ and
$\Lambda^{k}(J^{1}\pi)$. The $C^{\infty}$ module of $(r,s)$-type
tensor fields on $J^{1}\pi$ is denoted by
$\mathcal{T}_{s}^{r}(J^{1}\pi)$ and $\mathcal{T}(J^{1}\pi)$
denotes the tensor algebra on $J^{1}\pi$. By a vector valued
$l$-form ($l\geq0$) on $J^{1}\pi$ we mean an $(1,l)$-type tensor
field on $J^{1}\pi$ that is skew-symmetric in its $l$ arguments.

A parameterized curve on $M$ is a section of $\pi$: $\gamma:
\mathbb{R}\rightarrow M$, $\gamma(t)=(t,x^{i}(t))$. Its
\emph{first jet prolongation} $J^{1}\gamma:
t\in\mathbb{R}\rightarrow J^{1}\gamma(t)=\left(t,x^{i}(t),
{dx^{i}}/{dt}\right)\in J^{1}\pi$ is a section of the fibration
$\pi_{1}:=\pi\circ\pi_{10}:J^{1}\pi\rightarrow\mathbb{R}$.

Let $VJ^{1}\pi$ be the \emph{vertical subbundle} of $TJ^{1}\pi,$
$VJ^{1}\pi=\{\xi\in TJ^{1}\pi,\,\, D\pi_{10}(\xi)=0\}\subset
TJ^{1}\pi$. Its fibers, $V_{u}J^{1}\pi = \operatorname{Ker}
D_{u}\pi_{10}$, $u\in J^{1}\pi$, determine a regular,
$n$-dimensional \emph{vertical distribution}. The vertical
distribution is integrable, being tangent to the natural
foliation. Moreover, $VJ^{1}\pi=spann\{{\partial}/{\partial
y^{i}}\}$ and its annihilators are the \emph{contact 1-forms}
\begin{equation} \delta x^{i}=dx^{i}-y^{i}dt, \quad
i\in\{1,...,n\}. \label{eq:13}\end{equation} One can also view the
vertical subbundle as the image of the \emph{vertical
endomorphism} (or tangent structure)
\begin{equation} J=\frac{\partial}{\partial
y^{i}}\otimes\delta x^{i}. \label{eq:J}\end{equation} The vertical
endomorphism $J$ is a $(1,1)$-type tensor field, and hence a
vector valued 1-form on $J^{1}\pi$, with
$\operatorname{Im}J=VJ^{1}\pi$, $VJ^{1}\pi \subset
\operatorname{Ker}J$ and $J^{2}=0$.

A system \eqref{sode} of second order ordinary differential
equations, whose coefficients depend explicitly on time, can be
identified with a special vector field on $J^{1}\pi$, called a
semispray. A \emph{semispray} is a globally defined vector field
$S$ on $J^{1}\pi$ such that
\begin{equation} J(S)=0 \quad \textrm{and} \quad dt(S)=1. \label{eq:4}\end{equation}
Therefore, a semispray is a vector field $S$ on $J^{1}\pi$
characterized by the property that its integral curves are the
first jet prolongations of sections of $\pi_{1}:\,
J^{1}\pi\rightarrow\mathbb{R}$ . Locally, a semispray has the form
\begin{equation} S=\frac{\partial}{\partial t}+y^{i}\frac{\partial}{\partial
x^{i}}-2G^{i}(t,x,y)\frac{\partial}{\partial
y^{i}},\label{eq:5}\end{equation} where functions $G^{i}$, called
the \emph{semispray coefficients}, are locally defined on
$J^1\pi$.

A parameterized curve $\gamma: I\rightarrow M$ is called a
\emph{geodesic} of $S$ if $S\circ
J^{1}\gamma=\frac{d}{dt}(J^{1}\gamma).$ In local coordinates,
$\gamma(t)=(t,x^{i}(t))$ is a geodesic for the semispray $S$ given
by \eqref{eq:5} if and only if it satisfies the time dependent
system \eqref{sode} of SODE.

\subsection{Fr\"olicher-Nijenhuis theory on $J^{1}\pi$} \label{FN-theory}

In this section we give a short review of the
Fr\"olicher-Nijenhuis theory on $J^{1}\pi$, following
\cite{frolicher56, grifone00, KMS93, deleon89}.

Suppose that $A$ is a vector valued $l$-form on $J^{1}\pi$ and $B$
a (vector valued) $k$-form on $J^{1}\pi$. We can define an
\emph{algebraic derivation} \cite{KMS93} denoted by $i_{A}B$:
\begin{eqnarray*} \label{iaalpha} & & i_AB(X_1,\cdots, X_{k+l-1})=  \\
\nonumber &  & \frac{1}{l!(k-1)!}\sum_{\sigma\in S_{k+l-1}}
\operatorname{sign}(\sigma)\ B\left(A(X_{\sigma(1)},
\cdots,X_{\sigma(l)}), X_{\sigma(l+1)},\cdots, X_{\sigma(k+l-1)}
\right), \end{eqnarray*} where
$X_{1},\cdots,X_{k+l-1}\in{\mathfrak X}(J^{1}\pi)$ and $S_{p}$ is
the permutation group of elements $1,\cdots,p$ . We observe that
$i_{A}B$ is a $(k+l-1)$- form (vector valued form, respectively if
$B$ is vector valued) on $J^{1}\pi$. In \cite{frolicher56,
grifone00} this algebraic derivation is denoted by $B\barwedge A$
and it is called \emph{exterior inner product.}  For $B$ a scalar
form and $l=0$, $A$ is a vector field on $J^{1}\pi$ and $i_{A}B$
is the usual inner product of $k$-form $B$ with respect to vector
field $A$. When $l=1$, then $A$ is a $(1,1)$-type tensor field and
$i_{A}B$ is the $k$-form (or vector valued $k$-form, if $B$ is
vector valued)
\begin{eqnarray}
i_{A}B(X_{1},\cdots,X_{k}) & = & \sum_{i=1}^{k}B(X_{1},\cdots,\,
AX_{i},\cdots,\, X_{k}).\label{eq:inpr11}\end{eqnarray} For any
vector valued $l$-form $A$ on $J^{1}\pi$ we define $i_{A}B=0$, if
$k=0$, which means that $B\in C^{\infty}(J^{1}\pi)$ or
$B\in{\mathfrak X}(J^{1}\pi)$. If $k=1$, we have $i_{A}B=B\circ
A$.

Let $A$ be a vector valued $l$-for$m$ on $J^{1}\pi$, $l\geq0$. The
\emph{exterior derivative} with respect to $A$ is the map
$d_{A}:\Lambda^{k}(J^{1}\pi)\rightarrow\Lambda^{k+l}(J^{1}\pi),\,
k\geq0,$
\begin{eqnarray}
d_{A} & = & i_{A}\circ d-(-1)^{l-1}d\circ
i_{A}.\label{eq:36}\end{eqnarray} In \cite{KMS93}, the exterior
derivative $d_A$ is called the Lie derivative with respect to $A$
and it is denoted by $\mathcal{L}_{A}$. When $A\in{\mathfrak
X}(J^{1}\pi)$ and $k\geq0,$ we obtain $d_{A}=\mathcal{L}_{A}$, the
usual Lie derivative. In this case equation (\ref{eq:36}) is the
well known Cartan's formula. If $A=\operatorname{Id}$, the
identity $(1,1)$-type tensor field on $J^{1}\pi$, then
$d_{\operatorname{Id}}=d$, since
$i_{\operatorname{Id}}\alpha=k\alpha$ for
$\alpha\in\Lambda^{k}(J^{1}\pi)$.

Suppose $A$ and $B$ are vector valued forms on $J^{1}\pi$ of
degrees $l\geq0$ and $k\geq0$, respectively. Then, the
\emph{Fr\"olicher-Nijenhuis bracket }of $A$ and $B$ is the unique
vector valued $(k+l)$-form on $J^{1}\pi$ such that
\cite{frolicher56}
\begin{eqnarray} d_{[A,B]} & = & d_{A}\circ
d_{B}-(-1)^{kl}d_{B}\circ d_{A}.\label{eq:bracket}\end{eqnarray}
When $A$ and $B$ are vector fields, the Fr\"olicher-Nijenhuis
bracket coincides with the usual Lie bracket
$[A,B]=\mathcal{L}_{A}B.$

For a vector field $X\in{\mathfrak X}(J^{1}\pi)$ and a
$(1,1)$-type tensor field $A$ on $J^{1}\pi$, the
Fr\"olicher-Nijenhuis bracket $[X,A]=\mathcal{L}_{X}A$ is the
$(1,1)$-type tensor field \begin{equation}
\mathcal{L}_{X}A=\mathcal{L}_{X}\circ
A-A\circ\mathcal{L}_{X}.\label{eq:LXA}\end{equation}

The Fr\"olicher-Nijenhuis bracket of two $(1,1)$-type tensor
fields $A,\, B$ on $J^{1}\pi$ is the unique vector valued $2$-form
$[A,B]$ defined by \cite{KMS93}
\begin{eqnarray}
[A,B](X,Y) & = & [AX,BY]+[BX,AY]+(A\circ B+B\circ A)[X,Y]\nonumber \\
 &  & -A[X,BY]-A[BX,Y]-B[X,AY]\label{eq:br}\\
 &  & -B[AX,Y],\,\,\forall X,Y\in{\mathfrak X}(J^{1}\pi).\nonumber \end{eqnarray}
In particular, \begin{equation}
\frac{1}{2}[A,A](X,Y)=[AX,AY]+A^{2}[X,Y]-A[X,AY]-A[AX,Y].\label{eq:Nij}\end{equation}
 For a $(1,1)$-type tensor field $A$, the vector valued $2$-form
$N_{A}=\frac{1}{2}[A,A]$ is called the \emph{Nijenhuis tensor} of
$A$.

For the next commutation formulae on $\Lambda^{k}(J^{1}\pi)$,
$k\geq 0$, that will be used throughout the paper, we refer to
\cite[chapter 2]{grifone00}.
\begin{eqnarray}
i_{A}d_{B}-d_{B}i_{A} & = & d_{B\circ A}-i_{[A,B]},\label{eq:com1}\\
\mathcal{L}_{X}i_{A}-i_{A}\mathcal{L}_{X} & = & i_{[X,A]},\label{eq:com2}\\
i_{X}d_{A}+d_{A}i_{X} & = & \mathcal{L}_{AX}-i_{[X,A]},\label{eq:com3}\\
i_{A}i_{B}-i_{B}i_{A} & = & i_{B\circ A}-i_{A\circ
B},\label{eq:com4}\end{eqnarray} for $X\in{\mathfrak X}(J^{1}\pi)$
and $A, B\in\mathcal{T}_{1}^{1}(J^{1}\pi)$.

In this work we will also use the algebraic operator $A^{*}$,
\[ A^{*}B(X_{1},\cdots,X_{k})=B(AX_{1},\cdots,AX_{k}),\] for $A$ a
$(1,1)$-type tensor field and $B$ a (vector valued) $k$-form on
$J^1\pi$.

\subsection{Poincar\'e-type Lemma for semi-basic forms}
\label{semibasic}

For a vector valued $l$-form $A$, we say that a $k$-form $\omega$
on $J^{1}\pi$ is $d_{A}$-\emph{closed} if $d_{A}\omega=0$ and
$d_{A}$-\emph{exact} if there exists
$\theta\in\Lambda^{k-l}(J^{1}\pi)$ such that $\omega=d_{A}\theta.$
For a vector valued $1$-form $A$, from formulae \eqref{eq:bracket}
and \eqref{eq:Nij} we obtain that $d_A^2=d_{N_A}$. Therefore, if
$A$ is not integrable, which means that $N_A\neq 0$, $d_A$-exact
forms may not be $d_A$-closed.

Two forms  $\omega_{1}$ and $\omega_{2}$ on $J^{1}\pi$ are called
\emph{equivalent} (modulo $dt$) if $\omega_{1}\wedge
dt=\omega_{2}\wedge dt$.

For a vector valued $l$-form $A$, we say that a $k$-form
$\theta\in\Lambda^{k}(J^{1}\pi)$ is $d_{A}$-\emph{closed} (modulo
$dt$) if $d_{A}\theta\wedge dt=0$ and $d_{A}$-\emph{exact} (modulo
$dt$) if there exists $\omega$ in $\Lambda^{k-l}(J^{1}\pi)$ such
that $\theta\wedge dt=d_{A}\omega\wedge dt$.

For the vertical endomorphism $J$, its Fr\"olicher-Nijenhuis
tensor is given by
\begin{equation} N_{J}=-\frac{\partial}{\partial
y^{i}}\otimes\delta x^{i}\wedge dt=-J\wedge
dt.\label{eq:NijenhuisJ}\end{equation} Consequently, using formula
\eqref{eq:bracket}, we obtain $d_{J}^{2}=d_{N_J}=-d_{J\wedge
dt}\neq 0$. Therefore, $d_{J}$-exact forms on $J^{1}\pi$ may not
be $d_{J}$-closed. However, we will prove in the first part of
Lemma \ref{lem:(Poincare)} that $d_{J}$-exact (modulo $dt$) forms
are $d_{J}$-closed (modulo $dt$) forms. In the second part of
Lemma \ref{lem:(Poincare)} we will show that the converse is true
for semi-basic $k$-forms, only.

For a $k$-form $\omega$, the following identities can be obtained
immediately
\begin{eqnarray}
\nonumber  i_{J\wedge dt} \omega= (-1)^{k+1} i_J\omega \wedge dt, \\
\label{djdt} d_{J\wedge dt} \omega= (-1)^{k+1} d_J\omega \wedge
dt.
\end{eqnarray}

\begin{defn}
i) A $k$-form $\omega$ on $J^{1}\pi$, $k\geq1$, is called
\emph{semi-basic} if $\omega(X_{1},\cdots,X_{k})$$=0$, when one of
the vectors $X_{i}$, $i\in\{1,...,k\}$, is vertical.

ii) A vector valued $k$-form $A$ on $J^{1}\pi$ is called
\emph{semi-basic} if it takes values in the vertical subbundle and
$A(X_{1},\cdots,X_{k})=0$, when one of the vectors $X_{i}$, $
i\in\{1,...,k\}$ is vertical.
\end{defn}
A semi-basic $k$-form verifies the relation $i_{J}\theta=0$. The
converse is true only for $k=1$. Semi-basic $1$-forms are
annihilators for the vertical distribution and hence locally can
be expressed as
$\theta=\theta_{0}(t,x,y)dt+\theta_{i}(t,x,y)\delta x^{i}$.
Contact 1-forms $\delta x^{i}$ given by formula \eqref{eq:13} are
semi-basic 1-forms.

\begin{defn} \label{def:regtheta}
A semi-basic $1$-form $\theta$ on $J^1\pi$ is called
\emph{non-degenerate} if the $2$-form $d\theta+i_Sd\theta\wedge
dt$ has rank $2n$ on $J^1\pi$.
\end{defn}

If a vector valued $k$-form $A$ is semi-basic, then $J\circ A=0$,
$i_{J}A=0$ and $J^{*}A=0$. The converse is true only for $k=1$. It
follows that the vertical endomorphism $J$ is a vector valued,
semi-basic $1$-form.

Next lemma presents a characterization of forms that are
equivalent (modulo $dt$) to the null form and will be used to
prove a Poincar\'e-type Lemma for semi-basic forms.

\begin{lem} \label{omegadt}
 A $k$-form $\omega$ on $J^{1}\pi$ satisfies the condition
   $\omega\wedge dt=0$ if and only if it is of the form $\omega=i_{S\otimes dt}\omega=(-1)^{k+1}i_{S}\omega\wedge dt$,
   for an arbitrary semispray $S$.
\end{lem}
\begin{proof}
A simple computation gives $i_{S\otimes
dt}\omega=(-1)^{k+1}i_{S}\omega\wedge dt$. Therefore, if
$\omega=i_{S\otimes dt}\omega=(-1)^{k+1}i_{S}\omega\wedge dt$ it
follows that $\omega\wedge dt=0$. Conversely, if $\omega\wedge
dt=0$, we apply $i_{S}$ to this identity and get
$0=i_{S}\omega\wedge dt+(-1)^{k}\omega$. Hence,
$\omega=(-1)^{k+1}i_{S}\omega\wedge dt=i_{S\otimes dt}\omega$,
which completes the proof.
\end{proof}
\begin{lem}
\label{lem:(Poincare)}(Poincar\'e-type Lemma) Consider $\theta$ a
$k$-form on $J^{1}\pi$. \begin{itemize}
\item[i)] If $\theta$ is $d_J$-exact (modulo $dt$) then $\theta$
is $d_J$-closed (modulo $dt$).
\item[ii)] If $\theta$ is a semi-basic form and $d_{J}$-closed
(modulo $dt$), then $\theta$ is locally $d_{J}$-exact (modulo
$dt$).
\end{itemize}
\end{lem}
\begin{proof}
i) Suppose that $\theta\in\Lambda^{k}(J^{1}\pi)$ is $d_{J}$-exact
(modulo $dt$). Then, there exists
$\omega\in\Lambda^{k-1}(J^{1}\pi)$ such that $\theta\wedge dt
=d_{J}\omega\wedge dt$. If we apply $d_J$ to both sides of this
identity, and use formula \eqref{djdt}, we obtain
\begin{eqnarray*}
d_{J}\theta\wedge dt = d_{J}^{2}\omega\wedge dt = -d_{J\wedge
dt}\omega \wedge dt =0. \end{eqnarray*} which means that the
$k$-form $\theta$ is $d_J$-closed (modulo $dt$).

ii) We have to prove that for $\theta\in\Lambda^{k}(J^{1}\pi)$
semi-basic, with $d_{J}\theta\wedge dt=0$, there exists
$\omega\in\Lambda^{k-1}(J^{1}\pi)$ (locally defined), such that
$\theta\wedge dt=d_{J}\omega\wedge dt$. Since $\theta$ is
semi-basic, it follows that locally it has the form
\begin{eqnarray*}
\theta=\frac{1}{k!}\theta_{i_1...i_k}\delta x^{i_1}\wedge \cdots
\wedge \delta x^{i_k} +
\frac{1}{(k-1)!}\widetilde{\theta}_{i_1...i_{k-1}}\delta
x^{i_1}\wedge \cdots \wedge \delta x^{i_{k-1}}\wedge dt.
\end{eqnarray*}
Using the identities $ i_Jdt=0$, $i_J\delta x^i=0$, and
$i_Jdy^i=\delta x^i$, we obtain
\begin{eqnarray*}
d_J\theta\wedge dt = i_Jd\theta \wedge dt = \frac{1}{k!}
\frac{1}{(k+1)!} \varepsilon^{jj_1...j_k}_{i_1i_2...i_{k+1}}
\frac{\partial \theta_{j_1...j_k}}{\partial y^j} \delta
x^{i_1}\wedge \cdots \wedge \delta x^{i_{k+1}}\wedge dt,
\end{eqnarray*}
where $\varepsilon^{jj_1...j_k}_{i_1i_2...i_{k+1}}$ is Kronecker's
symbol. Hence
\begin{eqnarray*}
\frac{1}{(k+1)!} \varepsilon^{jj_1...j_k}_{i_1i_2...i_{k+1}}
\frac{\partial \theta_{j_1...j_k}}{\partial y^j}=0.
\end{eqnarray*} If one considers the transverse coordinates
$t,x^i$ of the natural foliation as parameters, one can use
Poincar\'e Lemma on the leafs of this foliation to obtain (locally
defined) functions $\omega_{j_i...j_{k-1}}(t,x^i,y^i)$ such that
\begin{eqnarray}
\theta_{i_1...i_k}=\frac{1}{(k-1)!}
\varepsilon^{jj_1...j_{k-1}}_{i_1i_2...i_{k}} \frac{\partial
\omega_{j_1...j_{k-1}}}{\partial y^j}. \label{tomega}
\end{eqnarray} Therefore, one can consider the (locally defined) semi-basic $(k-1)$-form
\begin{eqnarray*}
\omega=\frac{1}{(k-1)!}\omega_{i_1...i_{k-1}}\delta x^{i_1}\wedge
\cdots \wedge \delta x^{i_{k-1}}. \end{eqnarray*} In view of
identity \eqref{tomega} we have $\theta\wedge dt = d_J\omega
\wedge dt$, which means that $\theta$ is $d_J$-exact (modulo
$dt$).
\end{proof}

For semi-basic forms, the differential operator $d_{J}$ is closely
related to the exterior differential $d''$ along the leafs of the
natural foliation, studied by Vaisman in \cite{vaisman71}. This
relation as well as the proof of Theorem 3.1 from \cite{vaisman71}
can be used to give a different proof for Lemma
\ref{lem:(Poincare)}. This proof will require to fix a
distribution supplementary to the vertical distribution, which is
always possible, see Section \ref{nonlinear_connection}.

For semi-basic $1$ and $2$-forms a Poincar\'e-type Lemma and its
usefulness for the inverse problem of the calculus of variations
is discussed in Marmo et al. \cite{marmo89}. For forms along the
tangent bundle projection a similar result as in Lemma
\ref{lem:(Poincare)} has been shown in \cite[Prop. 2.1]{sarlet95}.

In Section \ref{section_lvf} we will use Poincar\'e-type Lemma
\ref{lem:(Poincare)} to find necessary and sufficient conditions
for a semi-basic $1$-form on $J^{1}\pi$ to coincide, or to be
equivalent (modulo $dt$), with the Poincar\'e-Cartan 1-form of
some Lagrangian function.

\section{Geometric objects induced by a semispray}
\label{gobjects}

In this section we present some geometric structures that can de
derived from a semispray using the Fr\"olicher-Nijenhuis theory:
nonlinear connection, Jacobi endomorphism, dynamical covariant
derivative. See \cite{cantrijn96, crampin96, deleon88, deleon89,
miron94, sarlet95}. These structures will be used later to express
necessary and sufficient conditions for a given semispray to be a
Lagrangian vector field.

\subsection{Nonlinear connection} \label{nonlinear_connection}

A \emph{nonlinear connection} on the first jet bundle $J^{1}\pi$
is an $(n+1)$-dimensional distribution $H:u\in J^1\pi \mapsto
H_u\subset T_uJ^1\pi$, supplementary to the vertical distribution
$VJ^{1}\pi$. This means that for each $u \in J^1\pi$, we have the
direct decomposition $T_uJ^1\pi=H_u\oplus V_u$.

A semispray $S$ induces a nonlinear connection on $J^{1}\pi$,
determined by the \emph{almost product structure}, or dynamical
connection \cite{deleon88}
\begin{equation} \Gamma=-\mathcal{L}_{S}J+S\otimes
dt, \quad \Gamma^{2}=\operatorname{Id}. \label{eq:6}\end{equation}
The \emph{horizontal projector} that corresponds to this almost
product structure is \begin{equation}
h=\frac{1}{2}\left(\operatorname{Id} - \mathcal{L}_{S}J+S\otimes
dt\right)\label{eq:7}\end{equation} and the \emph{vertical
projector} is $v=\operatorname{Id}-h.$ Note that in this paper we
chose to work with \emph{the weak horizontal projector} $h$
defined by formula \eqref{eq:7}. We can also consider
$h_{0}=S\otimes dt$ and $h_{1}=h-h_{0}$, the \emph{strong
horizontal projector}. The horizontal subspace is the eigenspace
corresponding to the eigenvalue $+1$ of $\Gamma,$ while the
vertical subspace is the eigenspace corresponding to the
eigenvalue $-1$. The horizontal subspace is spanned by $S$ and by
$$ \frac{\delta}{\delta x^{i}}=\frac{\partial}{\partial
x^{i}}-N_{i}^{j}\frac{\partial}{\partial y^{j}}, \quad
\textrm{where}  \quad N_{j}^{i}=\frac{\partial G^{i}}{\partial
y^{j}}.$$ From now on, whenever a semispray will be given, we
prefer to work with the following adapted basis and cobasis
\begin{equation}
\left\{ S,\frac{\delta}{\delta x^{i}},\frac{\partial}{\partial
y^{i}}\right\} , \quad \{dt,\delta x^{i},\delta y^{i}\},
\label{eq:11}\end{equation} with $\delta x^{i}=dx^{i}-y^{i}dt$ the
\emph{contact 1-forms} and
\begin{equation} \delta
y^{i}=dy^{i}+N_{j}^{i}dx^{j}+N_{0}^{i}dt,\quad
N_{0}^{i}=2G^{i}-N_{j}^{i}y^{j}.\label{eq:14}\end{equation}
Functions $N_{j}^{i}$ and $N_{0}^{i}$ are the coefficients of the
nonlinear connection induced by the semispray $S$. The $1$-forms
$\delta y^i$ are annihilators for the horizontal distribution.

With respect to the adapted basis and cobasis \eqref{eq:11} the
almost product structure can be locally expressed as
$$ \Gamma=S\otimes dt+\frac{\delta}{\delta x^{i}}\otimes\delta
x^{i}-\frac{\partial}{\partial y^{i}}\otimes\delta y^{i}.$$
Therefore, the horizontal and vertical projectors are locally
expressed as
$$ h=S\otimes dt+\frac{\delta}{\delta
x^{i}}\otimes\delta x^{i}, \quad v=\frac{\partial}{\partial
y^{i}}\otimes\delta y^{i}.$$ We consider the (1,1)-type tensor
field, which corresponds to the almost complex structure in the
autonomous case,
\begin{equation}
\mathbb{F}=h\circ\mathcal{L}_{S}h-J. \label{eq:21}\end{equation}
Tensor $\mathbb{F}$ satisfies $\mathbb{F}^{3}+\mathbb{F}=0,$ which
means that it is an $f(3,1)$ structure. It can be expressed
locally as
$$\mathbb{F}=\frac{\delta}{\delta x^{i}}\otimes\delta
y^{i}-\frac{\partial}{\partial y^{i}}\otimes\delta x^{i}.$$ For
some useful future calculus, we will also need the next formulae
\begin{eqnarray} \left[\frac{\delta}{\delta
x^{i}},\frac{\delta}{\delta x^{j}}\right] & = &
R_{ij}^{k}\frac{\partial}{\partial y^{k}},\quad
R_{jk}^{i}=\frac{\delta N_{j}^{i}}{\delta x^{k}}-\frac{\delta N_{k}^{i}}{\delta x^{j}},\label{eq:20} \\
\left[\frac{\delta}{\delta x^{i}},\frac{\partial}{\partial y^{j}}\right] & = &
\frac{\partial N_{i}^{k}}{\partial y^{j}}\frac{\partial}{\partial y^{k}}=
\frac{\partial^{2}G^{k}}{\partial y^{i}\partial y^{j}}\frac{\partial}{\partial y^{k}},\label{eq:20'}
\end{eqnarray}
for the Lie brackets of the vector fields of the adapted basis
\eqref{eq:11}. We also have:
\begin{eqnarray}
\mathcal{L}_{S}S=0, & \mathcal{L}_{S}
\displaystyle\frac{\delta}{\delta x^{i}}= N_{i}^{j}
\frac{\delta}{\delta x^{j}} + R_{i}^{j} \frac{\partial}{\partial
y^{j}}, & \mathcal{L}_{S}\frac{\partial}{\partial y^{i}} = -
\frac{\delta}{\delta x^{i}} + N_{i}^{j}\frac{\partial}{\partial y^{j}}, \label{eq:LSV}\\
\mathcal{L}_{S}dt=0, & \mathcal{L}_{S}\delta x^{i}=
-N_{j}^{i}\delta x^{j}+\delta y^{i}, & \mathcal{L}_{S}\delta y^{i}
=  - R_{j}^{i}\delta x^{j}-N_{j}^{i}\delta y^{j},
\label{eq:LSCV}\end{eqnarray} where
\begin{equation}
R_{j}^{i}=2\frac{\partial G^{i}}{\partial x^{j}}-\frac{\partial
G^{i}}{\partial y^{k}}\frac{\partial G^{k}}{\partial
y^{j}}-S\left(\frac{\partial G^{i}}{\partial y^{j}}\right).
\label{eq:18}\end{equation} $R_{j}^{i}$ are the components of a
$(1,1)$-type tensor field on $J^{1}\pi$, known as the \emph{second
invariant in KCC-theory} \cite{antonelli01}, \emph{the Douglas
tensor} \cite{douglas41, grifone00} or as the \emph{Jacobi
endomorphism} \cite{cantrijn96, crampin84a, jerie02, sarlet95}.

Using formulae \eqref{eq:J}, \eqref{eq:7} and the set of relations
\eqref{eq:LSV}, \eqref{eq:LSCV}, we obtain
\begin{eqnarray}
\mathcal{L}_{S}J & = & -\frac{\delta}{\delta x^{i}}\otimes\delta x^{i}+\frac{\partial}{\partial y^{i}}\otimes\delta y^{i},\label{eq:23}\\
\mathcal{L}_{S}h & = & \frac{\delta}{\delta x^{i}}\otimes\delta
y^{i}+R_{i}^{j}\frac{\partial}{\partial y^{j}}\otimes\delta
x^{i}.\label{eq:24}\end{eqnarray}

According to formulae \eqref{eq:br}, \eqref{eq:Nij},
\eqref{eq:20}, \eqref{eq:20'}, and
$[J,\Gamma]=2[J,h]-[J,\operatorname{Id}]=2[J,h]$ one can prove the
following proposition.
\begin{prop}
i) The weak torsion tensor field of the nonlinear connection
$\Gamma$ vanishes: $[J,h]=0,$ which is equivalent also with
$[J,\Gamma]=0$.

ii) The curvature tensor $R=N_h$ of the nonlinear connection
$\Gamma$ is a vector valued semi-basic 2-form, locally given by
\begin{equation}
R=\frac{1}{2}[h,h]=\frac{1}{2}R_{ij}^{k}\frac{\partial}{\partial
y^{k}}\otimes\delta x^{i}\wedge\delta
x^{j}+R_{i}^{j}\frac{\partial}{\partial y^{j}}\otimes
dt\wedge\delta x^{i},\label{eq:19}\end{equation} where
$R_{ij}^{k}$ and $R_{i}^{j}$ are given in formulae \eqref{eq:20}
and \eqref{eq:18}.
\end{prop}
The \emph{Jacobi endomorphism} can be defined as follows.
\begin{equation}
\Phi=v\circ\mathcal{L}_{S}h=\mathcal{L}_{S}h-\mathbb{F}-J.\label{eq:22}\end{equation}
It is a semi-basic, vector valued 1-form and satisfies
$\Phi^{2}=0$. Locally, it can be expressed as
\begin{eqnarray} \Phi=R_{i}^{j}\frac{\partial}{\partial
y^{j}}\otimes\delta x^{i}.\label{localphi}\end{eqnarray} As we can
see from formula \eqref{eq:19}, the Jacobi endomorphism is part of
the curvature tensor $R$, which is the \emph{third invariant in
KCC theory}, \cite{antonelli01}. Moreover, as we will see in the
next proposition, the curvature tensor can be obtained, using
Fr\"olicher-Nijenhuis theory, from the Jacobi endomorphism. The
result is similar to the one from \cite{sarlet95}, where different
techniques are used.

\begin{prop}
The Jacobi endomorphism and the curvature of the nonlinear
connection are related by the following formulae.
\begin{eqnarray}
\Phi & = & i_{S}R, \label{eq:23'} \\
\left[J,\Phi\right] & = & 3R + \Phi\wedge dt. \label{eq:24'}
\end{eqnarray}
\end{prop}
\begin{proof}
Since $2R=[h,h]$, for $X$ in ${\mathfrak X}(J^{1}\pi)$, we have $$
2(i_{S}R)(X) = [h,h](S,X) = 2[S,hX]-2h[S,hX]=2v[S,hX] =
2\Phi(X),$$ which proves formula \eqref{eq:23'}. To prove formula
\eqref{eq:24'} we use the local expression \eqref{eq:J} and
\eqref{localphi} of $J$ and $\Phi$ as well as the identity
$$ \frac{\partial R_{j}^{k}}{\partial y^{i}}-\frac{\partial R_{i}^{k}}{\partial
y^{j}}=3R_{ij}^{k},$$ which follows by a direct computation. The
Fr\"olicher-Nijenhuis bracket of $J$ and $\Phi$ is given by
\begin{eqnarray*}
[J,\Phi] & = & 2R_{i}^{j}\frac{\partial}{\partial y^{j}}\otimes
dt\wedge\delta x^{i} +\frac{1}{2}\left(\frac{\partial
R_{j}^{k}}{\partial y^{i}}-
\frac{\partial R_{i}^{k}}{\partial y^{j}}\right)\frac{\partial}{\partial y^{k}}\otimes\delta x^{i}\wedge\delta x^{j}\\
 & = & 2dt\wedge\Phi+\frac{3}{2}R_{ij}^{k}\frac{\partial}{\partial y^{k}}\otimes\delta x^{i}\wedge\delta x^{j}.\end{eqnarray*}
Using the above formula and formula \eqref{eq:19} that defines the
curvature $R$ we obtain formula \eqref{eq:24'}, which completes
the proof. \end{proof}

The following $(1,1)$-type tensor field $\Psi$ will allow us to
express in a simpler form the dynamical covariant derivative
induced by $S$, which will be discussed in the next section.
\begin{eqnarray}
\Psi & = & h\circ\mathcal{L}_{S}h+v\circ\mathcal{L}_{S}v=\Gamma\circ\mathcal{L}_{S}h = \mathbb{F}+J-\Phi,\label{eq:Psi}\\
\Psi & = & \frac{\delta}{\delta x^{i}}\otimes\delta
y^{i}-R_{i}^{j}\frac{\partial}{\partial y^{j}}\otimes\delta
x^{i}.\label{eq:Psilocal}\end{eqnarray}

\subsection{Dynamical covariant derivative} \label{dcv}

For a semispray $S$ there are various possibilities to define a
tensor derivation on $J^{1}\pi$. Such derivation was considered
first by Kosambi in \cite{kosambi35}, with the name of
biderivative. This derivation has been rediscovered latter and was
called the dynamical covariant derivative \cite{carinena91}. It
can be defined either as a derivation along the bundle projection
$\pi_{10}$, as in \cite{cantrijn96, crampin96, jerie02, mestdag01,
sarlet95} or as a derivation on the total space $TJ^1\pi$,
\cite{massa94, morando95}.

In this section we use Fr\"olicher-Nijenhuis theory to define the
dynamical covariant derivative as a tensor derivation on
$J^{1}\pi$, following the time independent case developed in
\cite{bucataru09}. In Proposition
\ref{pro:The-dynamical-covariant} we present some useful
commutation rules of the dynamical covariant derivative with the
geometric structures studied in the previous sections.

A map $\nabla: \mathcal{T}(J^{1}\pi)
\rightarrow\mathcal{T}(J^{1}\pi)$ is a \emph{tensor derivation} on
$\mathcal{T}(J^{1}\pi)$ if it satisfies the following conditions.
\begin{itemize}
\item[i)] $\nabla$ is $\mathbb{R}$-linear.
\item[ii)] $\nabla$ preserves the type of tensor fields.
\item[iii)] $\nabla$ obeys the Leibnitz rule: $\nabla(T\otimes
S)=\nabla T\otimes S+T\otimes\nabla S$, $\forall
T,S\in\mathcal{T}(J^{1}\pi)$.
\item[iv)] $\nabla$ commutes with any contractions.
\end{itemize}

For a semispray $S$ on $J^{1}\pi,$ we define the
$\mathbb{R}$-linear map $\nabla_{0}:{\mathfrak
X}(J^{1}\pi)\rightarrow{\mathfrak X}(J^{1}\pi)$
\begin{equation}
\nabla_{0}X=h[S,hX]+v[S,vX].\label{eq:25}\end{equation} It follows
that $\nabla_{0}(fX)=S(f)X+f\nabla_{0}X$, for all
$f\in\mathcal{C}^{\infty}(J^{1}\pi)$ and $X\in{\mathfrak
X}(J^{1}\pi)$.

Since any tensor derivation on $\mathcal{T}(J^{1}\pi)$ is
completely determined by its action over the smooth functions and
vector fields on $J^{1}\pi$, there exists a unique tensor
derivation $\nabla$ on $J^{1}\pi$ such that
$\nabla\vert_{\mathcal{C}^{\infty}(J^{1}\pi)}=S$ and $
\nabla\vert_{{\mathfrak X}(J^{1}\pi)}=\nabla_{0}$. This tensor
derivation is called the \emph{dynamical covariant derivative}
induced by the semispray $S$.

Next, we will obtain some alternative expressions for the action
of the dynamical covariant derivative $\nabla$ on ${\mathfrak
X}(J^{1}\pi)$, $\Lambda^{k}(J^{1}\pi)$ and $(1,1)$-type tensor
fields on $J^{1}\pi.$

From formula \eqref{eq:25} we obtain
$$\nabla\vert_{{\mathfrak X}(J^{1}\pi)}=h\circ\mathcal{L}_{S}\circ
h+v\circ\mathcal{L}_{S}\circ v = \mathcal{L}_{S} +
h\circ\mathcal{L}_{S}h+v\circ\mathcal{L}_{S}v.$$ Using the
$(1,1)$-type tensor field $\Psi$ defined in formula
\eqref{eq:Psi}, we obtain the following expression of the
dynamical covariant derivative.
\begin{equation} \nabla\vert_{{\mathfrak
X}(J^{1}\pi)}=\mathcal{L}_{S}+\Psi.\label{eq:26}\end{equation}

Since $\nabla$ satisfies the Leibnitz rule, we deduce that the
action of $\nabla$ on $k$-forms is given by
\begin{equation}
\nabla\vert_{\Lambda^{k}(J^{1}\pi)}=\mathcal{L}_{S}-i_{\Psi}.\label{eq:27}\end{equation}
Above formula \eqref{eq:27} implies that $\nabla$ is a degree zero
derivation on $\Lambda^{k}(J^{1}\pi)$. Therefore \cite[p.
69]{KMS93}, it can be uniquely written as a sum of a Lie
derivation, which is $\mathcal{L}_{S}$, and an algebraic
derivation, given by $i_{\Psi}$.

Similarly, we deduce that the action of $\nabla$ on a vector
valued $k$-form $A$ on $J^{1}\pi$ is given by
\begin{equation}
\nabla A = \mathcal{L}_{S}A + \Psi\circ
A-i_{\Psi}A.\label{eq:28}\end{equation}
\begin{prop}
\label{pro:The-dynamical-covariant} The dynamical covariant
derivative induced by a semispray $S$ has the following
properties.

i) $\nabla S=0$ and $\nabla i_{S}=i_{S}\nabla$.

ii) $\nabla h=\nabla v=0,$ which means that $\nabla$ preserves by
parallelism the horizontal and vertical distributions.

iii) $\nabla J=\nabla\mathbb{F}=0,$ which means that $\nabla$ acts
similarly on both horizontal and vertical distributions.

iv) The restriction of $\nabla$ to $\Lambda^{k}(J^{1}\pi)$ and the
exterior differential operator $d$ satisfy the commutation formula
\begin{equation}
d\nabla-\nabla d=d_{\Psi}.\label{eq:30}\end{equation}

v) The restriction of $\nabla$ to $\Lambda^{k}(J^{1}\pi)$
satisfies the commutation rule
\begin{equation} \nabla i_{A}-i_{A}\nabla=i_{\nabla
A}\label{eq:31}\end{equation} for a $(1,1)$- type tensor field $A$
on $J^{1}\pi.$ Hence the algebraic derivations with respect to $h,
v, J, \mathbb{F}$ commute with
$\nabla\vert_{\Lambda^{k}(J^{1}\pi)}.$
\end{prop}
\begin{proof}
First item follows directly using definition formula \eqref{eq:25}
and formula \eqref{eq:27}.

From the definition formula \eqref{eq:Psi} of tensor $\Psi$ we
obtain
\begin{equation} A\circ\Psi-\Psi\circ
A=\mathcal{L}_{S}A,\label{eq:32}\end{equation} for $A\in\{h, v, J,
\mathbb{F}\}.$ Using formula \eqref{eq:28}, it follows immediately
that $\nabla h=\nabla v=\nabla J=\nabla\mathbb{F}=0$ and hence we
proved item ii) and iii) of the proposition.

Using formula \eqref{eq:27} we obtain
\[ d\nabla=d\mathcal{L}_{S}-di_{\Psi}=\mathcal{L}_{S}d-i_{\Psi}d+d_{\Psi}=\nabla
d+d_{\Psi},\] for the restriction of $\nabla$ to
$\Lambda^{k}(J^{1}\pi)$. Therefore, formula \eqref{eq:30} is true.

From formulae \eqref{eq:28}, \eqref{eq:com2}, as well as
\eqref{eq:com4} we obtain
\[ \nabla i_{A}-i_{A}\nabla = \mathcal{L}_{S}i_{A} -
i_{A}\mathcal{L}_{S}-i_{\Psi}i_{A}+i_{A}i_{\Psi} =
i_{[S,A]}-i_{A\circ\Psi}+i_{\Psi\circ A}=i_{\nabla A},\] which
proves the last item of the proposition.
\end{proof}
First three items of Proposition \ref{pro:The-dynamical-covariant}
can be locally expressed as follows.
\begin{eqnarray}
\nabla S=0, &  & \nabla dt=0,\nonumber \\
\nabla\frac{\delta}{\delta x^{i}}=N_{i}^{j}\frac{\delta}{\delta x^{j}}, &  & \nabla\delta x^{i}=-N_{j}^{i}\delta x^{j},\label{eq:nablalocal}\\
\nabla\frac{\partial}{\partial
y^{i}}=N_{i}^{j}\frac{\partial}{\partial y^{j}}, &  & \nabla\delta
y^{i}=-N_{j}^{i}\delta y^{j}.\nonumber \end{eqnarray}

Tensor derivation $\nabla$ coincides with the dynamical covariant
derivative induced by the Berwald linear connection $\hat{\nabla}$
on $J^{1}\pi$, studied by Massa and Pagani in \cite{massa94}, in
the following sense $\nabla=\hat{\nabla}_{S}$. See also
\cite{mestdag01} for a detailed study of Berwald-type connection
associated to time dependent systems of SODE. The relation between
the dynamical covariant derivative and the Berwald connection
implies that tensor $\Psi$ is the shape map $A_{S}$ on the
manifold $N=J^{1}\pi,$ studied by Jerie and Prince in
\cite{jerie02} on an arbitrary manifold.

Another tensor derivation on ${\mathfrak X}(J^{1}\pi)$, induced by
a semispray, was proposed by Morando and Pasquero
\cite{morando95}. In our notations, Morando and Pasquero's
derivation can be expressed as $\nabla+\Psi = \mathcal{L}_S+ 2
\Psi$.

Next result is a technical lemma that expresses the action of the
dynamical covariant derivative on semi basic $1$-forms, and will
be useful for the proof of Theorem \ref{thm:LST2}.
\begin{lem} \label{lemnablatheta} For a semi-basic 1-form $\theta$ on
$J^{1}\pi,$ its dynamical covariant derivative can be expressed as
follows.
\begin{equation}
\nabla\theta=d_{h}i_{S}\theta+i_{S}d_{h}\theta.
\label{eq:nablatheta}\end{equation} Moreover, $\theta$ satisfies
the identity
\begin{equation}
i_{S}d_{h}\theta=i_{h}i_{S}d\theta.\label{eq:isdht}\end{equation}
\end{lem}
\begin{proof}
Using the identity
$d_{h}i_{S}+i_{S}d_{h}=\mathcal{L}_{S}-i_{[S,h]},$ one gets
\begin{equation}
\mathcal{L}_{S}\theta=d_{h}i_{S}\theta+i_{S}d_{h}\theta+i_{[S,h]}\theta.\label{eq:derLie}\end{equation}
Since $\theta$ is semi-basic we have
$i_{[S,h]}\theta=i_{\mathbb{F}+J+\Phi}\theta=i_{\mathbb{F}}\theta.$
Moreover,
$i_{\Psi}\theta=i_{\mathbb{F}+J-\Phi}\theta=i_{\mathbb{F}}\theta$
and hence
$\nabla\theta=\mathcal{L}_{S}\theta-i_{\Psi}\theta=d_{h}i_{S}\theta+i_{S}d_{h}\theta.$

For the second identity, we notice that
$i_{S}d_{h}\theta=i_{S}i_{h}d\theta-i_{S}d\theta.$ Using
$i_{S}i_{h}-i_{h}i_{S}=i_{hS}=i_{S},$ one gets
$i_{S}d_{h}\theta=i_{h}i_{S}d\theta$.\end{proof}

\subsection{Dual symmetries} \label{symmetries}

In Section \ref{sec_Helm} we will search for solutions of the
inverse problem of the calculus of variations in terms of
semi-basic $1$-forms. We prove that in the Lagrangian case there
is always a solution that is a $d_J$-closed semi-basic $1$-form,
which is the Poincar\'e-Cartan $1$-form of some Lagrangian
function. If the solution of the inverse problem contains a
semi-basic $1$-form that is not $d_J$-closed then it induces a
dual symmetry and a first integral of the semispray. For
discussions regarding adjoint symmetries in the context of the
inverse problem of the calculus of variations we refer to
\cite{carinena89, carinena91, sarlet90}. A detailed discussion of
symmetries, dual symmetries, adjoint symmetries and the relations
among them can be found in \cite{morando95}. In this section we
use expression \eqref{eq:27} of the dynamical covariant derivative
to characterize dual symmetries in terms of a Jacobi equation.

\begin{defn}
A 1-form $\omega$ on $J^{1}\pi$ is called a \emph{dual symmetry}
(or an invariant form) for a semispray $S$ if
$\mathcal{L}_{S}\omega=0$.
\end{defn}

In this work we will be interested in dual symmetries for which
$i_{S}\omega=0$. Since we will work with equivalence classes
(modulo $dt$) of dual symmetries, in each class we can choose a
representant of this form.

To express the condition $\mathcal{L}_{S}\omega=0$ for a $1$-form,
locally expressed as $\omega=\tilde{\omega}_{i}\delta
x^{i}+\omega_{i}\delta y^{i}$, we will use formulae \eqref{eq:27}
and \eqref{eq:nablalocal} for the dynamical covariant derivative.
Therefore $\mathcal{L}_S\omega=0$, which is equivalent to $\nabla
\omega=-i_{\Psi}\omega$ can be locally expressed as
\begin{eqnarray*}
\nabla\tilde{\omega}_{i}\delta x^{i}+\nabla\omega_{i}\delta
y^{i}=-\tilde{\omega}^{i}\delta y^{i}+R_{i}^{j}\omega_{j}\delta
x^{i} & \Leftrightarrow & \begin{cases}
\tilde{\omega}_{i} & =-\nabla\omega_{i},\\
\nabla\tilde{\omega}_{i}-R_{i}^{j}\omega_{j} &
=0.\end{cases}\end{eqnarray*} Hence $\omega$ is a dual symmetry,
if and only if $\omega=-\nabla\omega_{i}\delta
x^{i}+\omega_{i}\delta y^{i}$ and it satisfies the \emph{Jacobi
equation}
\begin{equation}
\nabla^{2}\omega_{i}+R_{i}^{j}\omega_{j}=0.
\label{eq:dualsymlocal}\end{equation} If $\omega$ is a dual
symmetry then $-i_{\Gamma}\omega$ is an adjoint symmetry. Locally,
an adjoint symmetry $\alpha$ with $i_{S}\alpha=0$ is locally
expressed as $\alpha=\nabla\omega_{i}\delta x^{i}+\omega_{i}\delta
y^{i}$ and verifies the same equation \eqref{eq:dualsymlocal}.

\section{Lagrangian vector fields}\label{section_lvf}

An approach to the inverse problem of the calculus of variations
seeks for the existence of a non-degenerate multiplier matrix
$g_{ij}(t,x,y)$ which relates the geodesic equations \eqref{sode}
of a semispray with the Euler-Lagrange equations
\eqref{el_equations} for a Lagrangian function $L$. In this case,
the Lagrangian function $L$ is determined from the condition that
the multiplier matrix $g_{ij}$ is the Hessian of $L$. Necessary
and sufficient conditions for the existence of such multiplier
matrix are known as Helmholtz conditions and were obtained, using
various techniques for both autonomous and nonautonomous case, in
\cite{anderson92, bucataru09, crampin81, henneaux82, kosambi35,
krupkova97, krupkova08, morandi90, santilli78, sarlet82}.

In our approach we look for solutions of the inverse problem in
terms of semi-basic 1-forms, see Theorem \ref{LVF} and Proposition
\ref{prop:LVF}. In the Lagrangian case, the Lagrangian function
$L$ is determined by the fact that its Poincar\'e-Cartan 1-form
coincides or it is equivalent (modulo $dt$) to a semi-basic 1-form
that satisfies certain conditions. We will call these conditions
\emph{Helmholtz conditions} and we will present how they lead to
the classic formulation of Helmholtz conditions in terms of a
multiplier matrix.

\subsection{Poincar\'e-Cartan $1$-forms}

\begin{defn} \label{def:regularl}
1) A smooth function $L\in C^{\infty}(J^{1}\pi)$ is called a \emph{Lagrangian
function}.

2) The \emph{Poincar\'e-Cartan 1-form} of the Lagrangian $L$ is
the semi basic $1$-form $\theta_{L}=Ldt+d_{J}L.$

3) If for a Lagrangian $L,$ the \emph{Poincar\'e-Cartan 2-form}
$d\theta_L$ has maximal rank $2n,$ then $L$ is called a
\emph{regular Lagrangian}.
\end{defn}
For a Lagrangian function $L,$ the Poincar\'e-Cartan 2-form
$d\theta_L$ can be written as follows, see \cite{crampin84a}
\begin{eqnarray} d\theta_L=\frac{\partial^{2}L}{\partial y^{i}\partial
y^{j}}\delta y^i \wedge \delta x^j.\label{dtl} \end{eqnarray}
Therefore, the Lagrangian function is regular if and only if the
$n\times n$ symmetric matrix with local components
\begin{equation} g_{ij}(t,x,y)=\frac{\partial^{2}L}{\partial
y^{i}\partial y^{j}}\label{eq:tensorg_ij}\end{equation} has rank
$n$ on $J^{1}\pi.$ Functions $g_{ij},$ given by formula
\eqref{eq:tensorg_ij}, are the components of a $(0,2)$-type
symmetric tensor $g=g_{ij}\delta x^{i}\otimes\delta x^{j}$, which
is called the \emph{metric tensor} of the Lagrangian function $L.$

Next lemma gives characterizations of Poincar\'e-Cartan 1-forms,
as well as equivalent (modulo $dt$) Poincar\'e-Cartan 1-forms, as
subsets of semi-basic 1-forms, using the exterior differential
operator $d_J$. In the time independent case, such
characterization appears in \cite{marmo89}.
\begin{lem}
\label{LemPoincare2} Let $\theta$ be a semi-basic 1-form on
$J^{1}\pi$.

i) $\theta$ is the Poincar\'e-Cartan 1-form of a Lagrangian
function if and only if $\theta$ is $d_{J}$-closed. Moreover, the
Lagrangian function is given by $L=i_{S}\theta.$

ii) $\theta$ is equivalent (modulo $dt$) to the Poincar\'e-Cartan
1-form of a Lagrangian function if and only if $\theta$ is
$d_{J}$-closed (modulo $dt$).
\end{lem}
\begin{proof}
i) If $\theta=Ldt+d_{J}L$ is the Poincar\'e-Cartan 1-form of a
Lagrangian function $L$, then it follows that $d_{J}\theta=0$ and
$L=i_{S}\theta.$

Conversely, let us assume that $\theta\in\Lambda^{1}(J^{1}\pi)$ is
semi-basic and $d_{J}\theta=0.$ Since
$d_{J}i_{S}+i_{S}d_{J}=\mathcal{L}_{JS}-i_{[S,J]}=i_{h-S\otimes
dt-v}$, we have that $d_{J}i_{S}\theta=\theta-i_{S}\theta dt$.
Therefore $\theta=i_{S}\theta dt+d_{J}i_{S}\theta$ and hence
$\theta$ is the Poincar\'e-Cartan 1-form of the Lagrangian
$L=i_{S}\theta.$

ii) We apply the Poincar\'e-type Lemma \ref{lem:(Poincare)} for
$k=1$. Suppose that the $1$-form $\theta$ is equivalent (modulo
$dt$) to the Poincar\'e-Cartan 1-form of the Lagrangian function
$L$. Therefore $\theta\wedge dt=d_JL\wedge dt$, which means that
$\theta$ is $d_J$-exact (modulo $dt$) and hence it is $d_J$-closed
(modulo $dt$).

Conversely, if the semi-basic $1$-form $\theta$ is $d_J$-closed
(modulo $dt$), it follows that there exists a (locally defined)
Lagrangian function $L$ such that $\theta \wedge dt=d_JL\wedge
dt=\theta_L\wedge dt$. Therefore, $\theta$ is equivalent (modulo
$dt$) with the Poincar\'e-Cartan 1-form of the Lagrangian function
$L$.
\end{proof}

Next Lemma will be applied in the next section in order to
formulate Helmholtz-type conditions for the semispray $S$.
\begin{lem} \label{lemaderLie} Let $\theta$ be a semi-basic, $d_J$-closed, 1-form on
$J^{1}\pi$. Then, the next equivalent conditions are satisfied.
\begin{itemize}
\item[i)] $\mathcal{L}_{S}\theta=di_{S}\theta+i_{S}d_{h}\theta$;
\item[ii)] $i_{S}d\theta=i_{S}d_{h}\theta$;
\item[iii)] $i_{S}d_{v}\theta=0$. \end{itemize}
\end{lem}
\begin{proof}
Consider $\theta$ a semi-basic 1-form on $J^{1}\pi$ such that
$d_{J}\theta=0.$ Using Lemma \ref{LemPoincare2}, we deduce
$i_{[S,h]}\theta=i_{\mathbb{F}}\theta=i_{\mathbb{F}}(i_{S}\theta
dt+d_{J}i_{S}\theta)=i_{\mathbb{F}}d_{J}i_{S}\theta$.

Since
$i_{\mathbb{F}}d_{J}-d_{J}i_{\mathbb{F}}=d_{J\circ\mathbb{F}}-i_{[\mathbb{F},J]}=d_{v}-i_{[\mathbb{F},J]},$
it results that $i_{\mathbb{F}}d_{J}i_{S}\theta =
d_{J}i_{\mathbb{F}}i_{S}\theta + d_{v}i_{S}\theta
-i_{[\mathbb{F},J]}i_{S}\theta = d_{v}i_{S}\theta$. Hence,
$i_{[S,h]}\theta=d_{v}i_{S}\theta$. If we substitute this in
formula \eqref{eq:derLie} we obtain
\begin{eqnarray*}
\mathcal{L}_{S}\theta & = & d_{h}i_{S}\theta+i_{S}d_{h}\theta+d_{v}i_{S}\theta \\
 & = & i_{h}di_{s}\theta+i_{v}di_{s}\theta+i_{S}d_{h}\theta\\
 & = & i_{\operatorname{Id}}di_{S}\theta+i_{S}d_{h}\theta = di_{S}\theta+i_{S}d_{h}\theta,\end{eqnarray*}
which proves condition i).

Evidently conditions $i)$ and $ii)$ are equivalent due to Cartan's
formula $\mathcal{L}_{S}=d\circ i_{S}+i_{S}\circ d$.

From $ii)$ and \eqref{eq:isdht}, we obtain
$i_{\operatorname{Id}}i_{S}d\theta =
i_{S}d\theta=i_{h}i_{S}d\theta$ and hence $i_{v}i_{S}d\theta=0$.
Using now the commutation rule $i_{S}i_{v}-i_{v}i_{S}=i_{vS}=0,$
we deduce that last two conditions are equivalent. \end{proof}

\subsection{Lagrangian semisprays and dual symmetries}

For a semispray $S,$ its \emph{geodesics}, given by the system
\eqref{sode} of SODE  coincide with the solutions of the
\emph{Euler-Lagrange equations} \eqref{el_equations} of a regular
Lagrangian $L$ if and only if the two sets of equations are
related by
\begin{equation}
g_{ij}\left(t,x,\frac{dx}{dt}\right)\left(\frac{d^{2}x^{j}}{dt^{2}}
+2G^{j}\left(t,x,\frac{dx}{dt}\right)\right) =
\frac{d}{dt}\left(\frac{\partial L}{\partial y^{i}}\right) -
\frac{\partial L}{\partial x^{i}},\label{eq:SEL}\end{equation}
with $g_{ij}$ given by formula \eqref{eq:tensorg_ij}. Therefore,
for a semispray $S$, there exists a regular Lagrangian function
$L$ that satisfies equation \eqref{eq:SEL} if and only if
\begin{equation} S\left(\frac{\partial L}{\partial
y^{i}}\right)-\frac{\partial L}{\partial
x^{i}}=0.\label{eq:geod1}\end{equation} Equations \eqref{eq:geod1}
can be globally expressed as \begin{equation}
\mathcal{L}_{S}\theta_{L}=dL.\label{eq:LSTL}\end{equation} In view
of Cartan's formula, above equation \eqref{eq:LSTL} is equivalent
to $i_Sd\theta_L=0$ and hence the regularity of the Lagrangian
function $L$ from Definition \ref{def:regularl} is equivalent to
the non-degeneracy of the Poincar\'e-Cartan $1$-form $\theta_L$
from Definition \ref{def:regtheta}. For details regarding
regularity aspects of Lagrangian systems see \cite[chapter
6]{krupkova97}.
\begin{defn}
A semispray $S$ is called a \emph{Lagrangian vector field} (or a
Lagrangian semispray) if and only if there exists a (locally
defined) regular Lagrangian $L$  that satisfies equation
\eqref{eq:LSTL}.
\end{defn}

Next theorem provides necessary and sufficient conditions, in
terms of a semi-basic $1$-form, for a semispray to be a Lagrangian
vector field.  It corresponds the characterizations in terms of
$2$-forms of Lagrangian vector field in \cite{anderson92,
balachandran80, crampin81, crampin84a, marmo89}. In
\cite{sarlet95}, Proposition 8.3 gives a characterization of a
Lagrangian vector field in terms of Poincar\'e-Cartan $1$-forms.
\begin{thm}
\label{LVF} A semispray $S$ is a Lagrangian vector field if and
only if there exists a non-degenerate, semi-basic $1$-form
$\theta$ on $J^1\pi$ such that $\mathcal{L}_S\theta$ is closed.
\end{thm}
\begin{proof}
We assume that the semispray $S$ is a Lagrangian vector field for
some regular Lagrangian function $L$. Since $L$ is regular it
follows that its Poincar\'e-Cartan 1-form $\theta_L=Ldt+d_{J}L$ is
non-degenerate. Moreover, $L$ satisfies equation \eqref{eq:LSTL},
which means that $\mathcal{L}_{S}\theta_L$ is exact and hence it
is a closed $1$-form.

For the converse, consider a non-degenerate, semi-basic 1-form
$\theta$ on $J^{1}\pi,$ such that $\mathcal{L}_{S}\theta$ is
closed. It follows that there exists a (locally defined)
Lagrangian function $L$ on $J^{1}\pi$ such that
\begin{equation}
\mathcal{L}_{S}\theta=dL.\label{eq:LSDL}\end{equation} If we apply
$i_{S}$ to both sides of formula \eqref{eq:LSDL} we obtain
\begin{equation} S(i_{S}\theta)=S(L).\label{eq:SiS}\end{equation}
Now, we apply $i_{J}$ to both sides of formula \eqref{eq:LSDL} and
obtain
\begin{equation}
i_{J}\mathcal{L}_{S}\theta=d_{J}L.\label{eq:iJLST}\end{equation}
Using
$i_{J}\mathcal{L}_{S}-\mathcal{L}_{S}i_{J}=-i_{[S,J]}=i_{h-S\otimes
dt-v}$, $i_{S\otimes dt}\theta=i_{S}\theta dt$ and the fact that
$\theta$ is a semi-basic 1-form, which means that
$i_{v}\theta=i_{J}\theta=0$ and $i_{h}\theta=\theta$, from formula
\eqref{eq:iJLST} we obtain \begin{equation} \theta=i_{S}\theta
dt+d_{J}L.\label{eq:37}\end{equation} From formula \eqref{eq:37}
it follows that $d\theta=d(i_S\theta)\wedge dt + dd_JL$, which can
be written as $d\theta+i_Sd\theta \wedge dt = \mathcal{L}_S\theta
\wedge dt + dd_JL=d\theta_L$. Hence the non-degeneracy of the
semi-basic $1$-form $\theta$ implies the regularity of the
Lagrangian function $L$. We will prove now that
$\mathcal{L}_{S}\theta_{L}=dL.$ Using formula \eqref{eq:37}, we
have $\theta_{L}=\theta+(L-i_{S}\theta)dt$ and hence, according to
formulae \eqref{eq:LSDL} and \eqref{eq:SiS}, we get
\begin{eqnarray*}
\mathcal{L}_{S}\theta_{L} & = &
\mathcal{L}_{S}\theta+\mathcal{L}_{S}(L-i_{S}\theta)dt =
dL,\end{eqnarray*} which completes the proof of the theorem.
\end{proof}

For a semispray $S$ let us introduce the set
\begin{eqnarray*}
\Lambda^1_S(J^1\pi)=\{\theta \in \Lambda^1(J^1\pi), \theta
\textrm{ semi-basic, non-degenerate, } \mathcal{L}_Sd\theta=0\}.
\label{lambda1s}
\end{eqnarray*}
Theorem \ref{LVF} states that a semispray $S$ is a Lagrangian
vector field if and only if $\Lambda^1_S(J^1\pi)\neq \emptyset$.
For a Lagrangian semispray $S$, on the set $\Lambda^1_S(J^1\pi)$
we introduce the following equivalence relation $\theta_1 \equiv
\theta_2$  if $\theta_1\wedge dt = \theta_2 \wedge dt$ and
$S\left(i_S(\theta_1-\theta_2)\right)=0$. It follows immediately
that $\theta_1 \equiv \theta_2$ if and only if there exists a
first integral $f$ for $S$ such that $\theta_1=\theta_2+fdt$.

From formulae \eqref{eq:LSDL} and \eqref{eq:37} it follows that
each equivalence class $[\theta]$  in $\Lambda^1_S(J^1\pi)$
contains exactly one Poincar\'e-Cartan $1$-form $\theta_L$ of some
regular Lagrangian $L$, where
\begin{eqnarray} \theta=\theta_{L}+(i_{S}\theta-L)dt. \label{istl} \end{eqnarray}
Therefore for each equivalence class $[\theta]$  in
$\Lambda^1_S(J^1\pi)$ there exists a unique Lagrangian $L$ such
that $[\theta]=[\theta_L]$. One can furthermore reformulate this
using Lemma \ref{LemPoincare2}, each equivalence class in
$\Lambda^1_S(J^1\pi)$ contains exactly one $d_J$-closed semi-basic
$1$-form.

If an equivalence class $[\theta_L]$ contains a semi-basic
$1$-form $\theta$, which is not $d_J$-closed, then $i_S\theta-L$
is a first integral and $i_Sd\theta$ is a dual symmetry for $S$.
In this case the Lagrangian vector field $S$ is a conservative
vector field.

According to the above discussion, we can strengthen the
conclusion of Theorem \ref{LVF} as follows.
\begin{prop} \label{prop:LVF} Consider $S$ a semispray.
\begin{itemize} \item[i)] $S$ is a Lagrangian vector field if and only if there
exists $\theta \in \Lambda^1_S(J^1\pi)$ such that $d_J\theta=0$.
In this case $\theta=\theta_L$, for some locally defined
Lagrangian function $L$, and $L_S\theta=dL$.
\item[ii)] $S$ is a conservative Lagrangian vector field if and only if there
exists $\theta \in \Lambda^1_S(J^1\pi)$ such that
$d_J\theta\neq0$. In this case $\theta=\theta_L+fdt$, for some
locally defined Lagrangian function $L$ and $f$ a conservation law
for $S$.
\end{itemize}
\end{prop} Necessary and sufficient conditions for the existence
of semi-basic $1$-forms that satisfy either condition i) or
condition ii) of Proposition \ref{prop:LVF} will be discussed in
the next section in Theorems \ref{thm:LST1} and \ref{thm:LST2}.
However, we want to emphasize that a semi-basic $1$-form $\theta
\in \Lambda^1_S(J^1\pi)$ which is not $d_J$-closed gives rise to a
Lagrangian function, a first integral and a dual symmetry of the
semispray, see Theorem \ref{thm:LST2}.

\section{Helmholtz-type conditions} \label{sec_Helm}

In this section we use Fr\"olicker-Nijenhuis theory on $J^{1}\pi$
and geometric objects associated to a semispray $S$, to obtain
invariant conditions for a semi-basic 1-form $\theta$ on
$J^{1}\pi$ that are equivalent with the condition that
$\mathcal{L}_{S}\theta$ is a closed $1$-form. Therefore, in view
of Theorem \ref{LVF}, we obtain necessary and sufficient
conditions, in terms of a semi-basic $1$-form, for a semispray $S$
to be a Lagrangian vector field. We will relate these conditions
with the classic formulation of Helmholtz conditions in terms of a
multiplier matrix in the next section.

\subsection{Semi-basic $1$-forms, symmetries, and Helmholtz-type conditions}
In this section we present two theorems that give
characterizations for those semi-basic $1$-forms in the set
$\Lambda^1_S(J^1\pi)$. First theorem seeks for a solution of the
inverse problem on a restricted class of semi-basic $1$-forms:
those forms that are $d_{J}$-closed, and hence represent the
Poincar\'e-Cartan forms for some Lagrangian functions on
$J^{1}\pi$. Second theorem seeks for a solution of the inverse
problem on a larger class of semi-basic $1$-forms, which are not
$d_J$-closed. It is important to note that in this case, if there
is a solution then it induces a dual symmetry and a first integral
of the given semispray.

\begin{thm}
\label{thm:LST1} Let $\theta$ be a semi-basic, $d_J$-closed,
1-form on $J^{1}\pi$. Then, the following conditions are
equivalent
\begin{itemize}
\item[i)] $\mathcal{L}_{S}\theta$ is closed;
\item[ii)] $\mathcal{L}_{S}\theta$ is exact;
\item[iii)] $\mathcal{L}_{S}\theta=di_{S}\theta$;
\item[iv)] $d_{h}\theta=0$.
\end{itemize}
\end{thm}
\begin{proof}
Implications $iii)\Rightarrow ii)\Rightarrow i)$ are immediate.
Therefore it remains to prove implications $iv)\Rightarrow iii$)
and $i)\Rightarrow iv)$.

We prove first that condition $iv)$ implies condition $iii).$
Consider $\theta$ a semi-basic 1-form on $J^{1}\pi$ such that
$d_{J}\theta=d_{h}\theta=0.$ According to Lemma \ref{lemaderLie}
we have that $d_{J}\theta=0$ implies
$\mathcal{L}_{S}\theta=di_{S}\theta+i_{S}d_{h}\theta$. Since
$d_{h}\theta=0$ it results $\mathcal{L}_{S}\theta=di_{S}\theta$,
which is condition $iii)$.

Finally, we have to prove implication $i)\Rightarrow iv).$ If
$\mathcal{L}_{S}\theta$ is closed, we have that
$\mathcal{L}_{S}d\theta=0$ and hence
$i_{J}\mathcal{L}_{S}d\theta=0$. Using the commutation rule
$i_{J}\mathcal{L}_{S}-\mathcal{L}_{S}i_{J}= i_{\Gamma-S\otimes
dt}$ and formula $i_{J}d\theta=d_{J}\theta=0$ it results
\begin{eqnarray}
i_{\Gamma-S\otimes dt}d\theta=0. \label{igsdt} \end{eqnarray} Now,
we have $i_{\Gamma}d\theta=i_{2h-\operatorname{Id}}d\theta =
2i_{h}d\theta-i_{\operatorname{Id}}d\theta=2(i_{h}d\theta-d\theta)=2(i_{h}d\theta-di_{h}\theta)=2d_{h}\theta.$
If we substitute this in formula \eqref{igsdt} we obtain
\begin{equation} i_{S\otimes
dt}d\theta=2d_{h}\theta.\label{eq:(2)}\end{equation}

If we apply $i_{S}$ to both sides of equation \eqref{igsdt} and
use the identity $i_{S}i_{\Gamma-S\otimes dt}-i_{\Gamma-S\otimes
dt}i_{S}=0$, we obtain
\begin{equation}
i_{\Gamma-S\otimes dt}i_{S}d\theta=0.\label{eq:(1)}\end{equation}

If we apply $i_{S}$ to both sides of the equation \eqref{eq:(2)},
use condition ii) from Lemma \ref{lemaderLie}, and commutation
formula $i_{S}i_{S\otimes dt}-i_{S\otimes dt}i_{S}=i_{S}$ we
obtain $2i_{S}d_h\theta=i_{S\otimes dt}i_{S}d\theta+i_{S}d\theta$,
which implies $i_{\operatorname{Id}}(i_{S}d\theta)=i_{S\otimes
dt}i_{S}d\theta$. Last formula is equivalent to
\begin{equation}
i_{\operatorname{Id}-S\otimes dt}i_{S}d\theta=0.
\label{eq:(3)}\end{equation} Equations \eqref{eq:(1)} and
\eqref{eq:(3)} lead to \begin{equation}
i_{h}i_{S}d\theta=0.\label{eq:(4)}\end{equation} Condition iii)
from Lemma \ref{lemaderLie} can be written as \begin{equation}
i_{v}i_{S}d\theta=0.\label{eq:(5)}\end{equation} Using equations
\eqref{eq:(4)} and \eqref{eq:(5)} we obtain $i_{S}d\theta=0$.
Since $i_{S\otimes dt}d\theta=(i_{S}d\theta)\wedge dt=0,$ from
\eqref{eq:(2)} we obtain that $d_{h}\theta=0$, which is condition
iv). \end{proof}

Condition $iii)$ is equivalent with condition $i_{S}d\theta=0$
required in all previous works that deals with the inverse problem
of the calculus of variations for the time dependent case
\cite{anderson92, crampin84a, marmo89, sarlet95}. For example
condition $i_{S}d\theta=0$ is reflected in the expression of the
$2$-form $\omega$ in formula (3.2) in \cite{anderson92}. The next
theorem deals with a larger class of semi-basic 1-forms, where
condition $i_{S}d\theta=0$ is not required. In this case, in view
of Proposition \ref{prop:LVF}, we obtain a characterization for
conservative Lagrangian vector fields.

\begin{thm}
\label{thm:LST2} Let $\theta$ be a semi-basic 1-form on $J^1\pi$,
which is not $d_J$-closed. Then $\mathcal{L}_{S}\theta$ is a
closed $1$-form if and only if the following conditions hold true.
\begin{itemize}
\item[$(H_{1})$] $d_{J}\theta\wedge dt=0$ ($\theta$ is $d_{J}$-closed
modulo $dt$);
\item[$(H_{2})$] $d_{h}\theta\wedge dt=0$ ($\theta$ is $d_{h}$-closed
modulo $dt$);
\item[$(H_{3})$] $d_{\Phi}\theta\wedge dt=0$ ($\theta$ is $d_{\Phi}$-closed
modulo $dt$);
\item[$(H_{4})$] $\nabla d\theta\wedge dt=0$;
\item[$(DS)$] The 1-form $i_{S}d\theta$ is a dual symmetry for
$S.$\end{itemize}
\end{thm}
\begin{proof}
Suppose that $\mathcal{L}_{S}\theta$ is a closed $1$-form. If we
apply $i_J$ to $\mathcal{L}_Sd\theta=0$ and use commutation rule
\eqref{eq:com2} we obtain
\begin{eqnarray}
\mathcal{L}_{S}i_{J}d\theta+i_{\Gamma-S\otimes dt}d\theta=0.
\label{lsijdt}
\end{eqnarray}
If we apply again $i_J$ to above formula and use commutation rule
\eqref{eq:com2} we obtain
\begin{eqnarray*}
\mathcal{L}_{S}i_{J}d_{J}\theta+i_{\Gamma-S\otimes
dt}i_{J}d\theta+i_{J}i_{\Gamma-S\otimes dt}d\theta=0.
\end{eqnarray*}
Since $\theta$ is semi-basic it follows that
$i_Jd_J\theta=i^2_Jd\theta=2!d\theta \circ J^*=0$. Using
commutation rule \eqref{eq:com4}, $J\circ (\Gamma-S\otimes dt)=J$,
and $(\Gamma-S\otimes dt)\circ J=-J$, we get $$0 =
2i_{J}i_{\Gamma-S\otimes dt}d\theta + i_{J\circ(\Gamma-S\otimes
dt)} d\theta - i_{(\Gamma-S\otimes dt)\circ J} d\theta  =
i_{J}i_{\Gamma-S\otimes dt}d\theta+i_{J}d\theta.$$ Therefore, we
have \begin{equation}
d_{J}\theta+i_{J}i_{\Gamma}d\theta-i_{J}i_{S\otimes
dt}d\theta=0.\label{eq:DJTeta1}\end{equation} Since $
i_{\Gamma}d\theta=2d_{h}\theta$ is a semi-basic 2-form, it follows
that $i_{J}i_{\Gamma}d\theta=0$. Also, $i_{J}i_{S\otimes
dt}d\theta=i_{S\otimes dt}i_{J}d\theta=i_{S\otimes
dt}d_{J}\theta$. Formula \eqref{eq:DJTeta1} becomes
\begin{eqnarray} d_{J}\theta=i_{S\otimes
dt}d_{J}\theta.\label{eq:ih0djt}\end{eqnarray} Applying Lemma
\ref{omegadt} we obtain $d_{J}\theta\wedge dt=0$, which is
condition $(H_{1})$.

If we use formulae \eqref{eq:ih0djt}, \eqref{lsijdt} and
commutation rule $\mathcal{L}_Si_{S\otimes dt}= i_{S\otimes dt}
\mathcal{L}_S$, we obtain $i_{\Gamma}d\theta=i_{S\otimes
dt}(d\theta-\mathcal{L}_Sd_J\theta)$. Therefore,
$i_{\Gamma}d\theta\wedge dt=0$ and hence $d_h\theta\wedge dt=0$,
which is condition $(H_{2})$.

From the action of the dynamical covariant derivative $\nabla$,
expressed by formula \eqref{eq:27}, on $2$-forms we obtain
\begin{equation} \mathcal{L}_{S}d\theta=\nabla
d\theta+i_{\Psi}d\theta.\label{eq:nabladteta}\end{equation}
Applying $i_{\Gamma}$ to both sides of this identity and using
$\nabla i_{\Gamma}=i_{\Gamma}\nabla$, we obtain $ \nabla
i_{\Gamma}d\theta+i_{\Gamma}i_{\Psi}d\theta = 0$, which in view of
commutation formula \eqref{eq:com4} is equivalent to $\nabla
i_{\Gamma}d\theta+i_{\Psi}i_{\Gamma}d\theta+i_{\Psi\circ\Gamma-\Gamma\circ\Psi}d\theta=0$.
From Proposition \ref{pro:The-dynamical-covariant} we obtain
$\nabla\Gamma=0$, which implies
$\Psi\circ\Gamma-\Gamma\circ\Psi=-\mathcal{L}_{S}\Gamma=-2\mathcal{L}_{S}h
 =  -2(\mathbb{F}+J+\Phi)$ and hence
\begin{eqnarray}
\nabla i_{\Gamma}d\theta +
i_{\Psi}i_{\Gamma}d\theta-2i_{\mathbb{F}+J+\Phi}d\theta & = &
0.\label{eq:nablaigamadteta}\end{eqnarray} Using formula
\eqref{eq:nablaigamadteta} and $\nabla i_{\Gamma}d\theta\wedge
dt=i_{\Psi}i_{\Gamma}d\theta\wedge dt=0$ we obtain
$(i_{\mathbb{F}+J+\Phi}d\theta)\wedge dt=0$.

From second formula \eqref{eq:nabladteta} it results that $(\nabla
d\theta+i_{\mathbb{F}+J-\Phi}d\theta)\wedge dt=0$. Therefore
\begin{equation}
\begin{cases}
(\nabla d\theta+2i_{\mathbb{F}+J}d\theta)\wedge dt & =0,\\
(\nabla d\theta-2i_{\Phi}d\theta)\wedge dt &
=0.\end{cases}\label{eq:nabla1}\end{equation} From formula
\eqref{eq:nabla1}, it follows that there exists a 1-form $\omega$
on $J^{1}\pi$ such that $\nabla
d\theta-2i_{\Phi}d\theta=\omega\wedge dt$. Applying $i_{h}$ to
this identity, we get \begin{equation} \nabla
i_{h}d\theta-2i_{h}i_{\Phi}d\theta=i_{h}\omega\wedge
dt+\omega\wedge dt.\label{eq:nablaih}\end{equation}

We also have
$i_{h}i_{\Phi}d\theta=i_{\Phi}i_{h}d\theta+i_{\Phi\circ
h-h\circ\Phi} d\theta=i_{\Phi}i_{h}d\theta+i_{\Phi}d\theta$.

Since $i_{h}d\theta\wedge dt=d\theta\wedge dt$ $(H_{2})$, it exists
a 1-form $\tilde{\omega}$ on $J^{1}\pi$ such that $i_{h}d\theta=d\theta+\tilde{\omega}\wedge dt$.
It results that \begin{equation}
i_{h}i_{\Phi}d\theta=2i_{\Phi}d\theta+i_{\Phi}\tilde{\omega}\wedge dt.\label{eq:ihiphi}\end{equation}

Formulae \eqref{eq:nablaih} and \eqref{eq:ihiphi} lead to
\begin{equation} (\nabla d\theta-4i_{\Phi}d\theta)\wedge
dt=0.\label{eq:nabla2}\end{equation}

Relations \eqref{eq:nabla1} and \eqref{eq:nabla2} imply
$i_{\Phi}d\theta\wedge dt=0$, which is condition $(H_{3})$ as well
as $\nabla d\theta\wedge dt=0$, which is condition $(H_{4})$.

Evidently, the assumption $\mathcal{L}_{S}\theta$ closed implies
$\mathcal{L}_{S}i_{S}d\theta=i_{S}\mathcal{L}_{S}d\theta=0$, and
hence condition $(DS)$ is also satisfied.

For the converse, we assume that there exists a semi-basic 1-form
$\theta$ on $J^{1}\pi$ such that conditions $(H_{1})-(H_{4})$ and
$(DS)$ are verified. We will prove that $\mathcal{L}_{S}\theta$ is
closed.

Using condition $(DS)$ it follows that $i_S(\mathcal{L}_Sd\theta
\wedge dt)=\mathcal{L}_S(i_Sd\theta)\wedge dt +
\mathcal{L}_Sd\theta = \mathcal{L}_Sd\theta.$ Therefore, it
suffices to show that $\mathcal{L}_Sd\theta \wedge dt=0$. Using
formula \eqref{eq:27} and condition $(H_{4})$ we obtain
$$ \mathcal{L}_Sd\theta \wedge dt=(\nabla d\theta
+i_{\Psi}d\theta)\wedge dt = i_{\Psi}d\theta \wedge dt.$$ We will
show that the $2$-form $i_{\Psi}d\theta \wedge dt$ vanishes by
evaluating it on pairs of strong horizontal and/or vertical vector
fields.

From conditions $(H_{1})-(H_{3})$ we obtain that the $2$-forms
$d_J\theta$, $d_h\theta$ and $d_{\Phi}\theta$ vanish on any pair
of strong horizontal vector fields $h_1X, h_1Y$. From condition
$(H_{1})$ we have $i_Jd\theta(h_1X, h_1Y)=0$ and hence
\begin{eqnarray}
d\theta(JX, h_1Y)+ d\theta(h_1X, JY)=0. \label{dtjh}
\end{eqnarray}
From condition $(H_{2})$ it follows that $(i_hd\theta -
d\theta)(h_1X, h_1Y)=0$, which implies
\begin{eqnarray}
d\theta(h_1X, h_1Y)=0. \label{dthh}
\end{eqnarray}
From condition $(H_{3})$ we obtain that $i_{\Phi}d\theta(h_1X,
h_1Y)=0$, which implies
\begin{eqnarray}
d\theta(\Phi X, h_1Y)+ d\theta(h_1X, \Phi Y)=0. \label{dtph}
\end{eqnarray}

Using formulae \eqref{dtjh}, \eqref{dthh} and \eqref{dtph} we
obtain
\begin{eqnarray*}
(i_{\Psi}d\theta)(JX, JY) & = & d\theta(JX, h_1Y)+ d\theta(h_1X,
JY)=0,
\nonumber \\
(i_{\Psi}d\theta)(h_1X, h_1Y) & = & d\theta(-\Phi X, h_1Y)+
d\theta(h_1X, -\Phi Y)=0, \label{ipdtjj} \\
(i_{\Psi}d\theta)(JX, h_1Y) & = & d\theta(h_1X, h_1Y)+ d\theta(JX,
-\Phi Y)=0, \nonumber
\end{eqnarray*}
for any arbitrary pair of vector fields $X, Y$ on $J^1\pi$. Last
three formulae imply that $i_{\Psi} d\theta \wedge dt=0$ and hence
$\mathcal{L}_Sd\theta=0$.
\end{proof}
In view of Proposition \ref{prop:LVF}, one can reformulate Theorem
\ref{thm:LST2} as follows. A semispray $S$ is a Lagrangian vector
field if and only there exists a semi-basic $1$-form $\theta$ that
satisfies the Helmholtz conditions $(H_{1})-(H_{4})$. Note that in
this case we might have $d_J\theta=0$ and hence $\theta=\theta_L$
for some Lagrangian function $L$. If $d_J\theta\neq 0$ then
$\theta$ is equivalent (modulo $dt$) with the Poincar\'e-Cartan
$1$-form of some Lagrangian function. In this case, $S$ is a
conservative Lagrangian vector field and $i_Sd\theta$ is a dual
symmetry.

\subsection{Semi-basic 1-forms and multiplier matrices}
In this section we present a proof of the Theorem \ref{thm:LST2}
using local coordinates. First, this allows to show the usefulness
of the covariant derivative studied in Section \ref{dcv}, as well
as the use of the adapted basis and cobasis \eqref{eq:11}.
Secondly, it will relate the Helmholtz-type conditions
$(H_{1})-(H_{4})$ presented in Theorem \ref{thm:LST2} with their
classic formulation for a multiplier matrix.

Consider that $\theta=\theta_{0}dt+\theta_{i}\delta x^{i}$ is a
semi-basic 1-form on $J^{1}\pi$ that is not $d_J$-closed and such
that the $1$-form $\mathcal{L}_{S}\theta$ is closed. For the
$1$-form $\theta$ we will use the following
notations.\begin{eqnarray}
a_{i}=\frac{\partial\theta_{0}}{\partial y^{i}}-\theta_{i}, &  & b_{i}=\frac{\delta\theta_{0}}{\delta x^{i}}-\nabla\theta_{i},\label{eq:40}\\
b_{ij}=\frac{\delta\theta_{i}}{\delta
x^{j}}-\frac{\delta\theta_{j}}{\delta x^{i}}, &  &
g_{ij}=\frac{\partial\theta_{i}}{\partial y^{j}}.\nonumber
\end{eqnarray} A direct calculus using adapted cobasis \eqref{eq:11} leads to \begin{eqnarray} d\theta
& = & b_{i}\delta x^{i}\wedge dt+a_{i}\delta y^{i}\wedge
dt+\frac{1}{2}b_{ij}\delta x^{j}\wedge\delta x^{i}+g_{ij}\delta
y^{j}\wedge\delta x^{i}.\label{eq:41}\end{eqnarray} We emphasize
the presence of the two terms $a_i$ and $b_i$ in formula
\eqref{eq:41} due to the fact that $i_Sd\theta \neq 0$, terms
which do not appear in previous work for the time dependent case
of the calculus of variations, see \cite{anderson92, crampin84a,
sarlet95}.

Using formula \eqref{eq:27} for the dynamical covariant
derivative, the $2$-form $\mathcal{L}_Sd\theta=\nabla d\theta
+i_{\Psi}d\theta$ can be expressed in terms of the adapted cobasis
\eqref{eq:11} as follows.
\begin{eqnarray}
\mathcal{L}_{S}d\theta & = & (\nabla b_{i}-a_{j}R_{i}^{j})\delta x^{i}\wedge dt+(b_{i}+\nabla a_{i})\delta y^{i}\wedge dt\nonumber \\
 &  & +\frac{1}{2}(\nabla b_{ij}-g_{ik}R_{j}^{k}+g_{jk}R_{i}^{k})\delta x^{j}\wedge\delta x^{i}\label{eq:42}\\
 &  & +(\nabla g_{ij}+b_{ij}-b_{ji})\delta y^{j}\wedge\delta x^{i}+\frac{1}{2}(g_{ij}-g_{ji})\delta y^{j}\wedge\delta y^{i}.\nonumber \end{eqnarray}
From above formula \eqref{eq:42} it follows that the condition
$\mathcal{L}_{S}\theta$ is closed is equivalent with the following
two sets of conditions that correspond to Helmholtz conditions
($H_1$)-($H_4$) in Theorem \ref{thm:LST2}
\begin{eqnarray}
g_{ij} & = & g_{ji},\label{eq:43}\\
g_{ik}R_{j}^{k} & = & g_{jk}R_{i}^{k},\label{eq:44}\\
\nabla g_{ij} & = & 0,\label{eq:45}\\
b_{ij} & = & 0,\label{eq:46}
\end{eqnarray}
and respectively to condition (DS) in Theorem \ref{thm:LST2}
\begin{eqnarray}
b_{i}+\nabla a_{i} & = & 0,\label{eq:47}\\
\nabla b_{i}-a_{j}R_{i}^{j} & = & 0.\label{eq:48}\end{eqnarray} If
the above two equations \eqref{eq:46} and \eqref{eq:47} hold good
we deduce
\begin{equation}
\nabla^{2}a_{i}+a_{j}R_{i}^{j}=0\label{eq:49}\end{equation} which
is equivalent with the fact that the 1-form
\[i_{S}d\theta=\nabla a_{i}\delta x^{i}-a_{i}\delta
y^{i}=-b_{i}\delta x^{i}-a_{i}\delta y^{i}\] is a dual symmetry of
the semispray $S$, since it satisfies equation
\eqref{eq:dualsymlocal}.

Next, we present the local expressions of the $2$-forms
$d_J\theta$, $d_h\theta$, $d_{\Phi}\theta$ and $\nabla d\theta$.
\begin{eqnarray*}
d_{J}\theta & = & a_{i}\delta x^{i}\wedge dt + \frac{1}{2}(g_{ij}-g_{ji})\delta x^{j}\wedge\delta x^{i},\nonumber \\
d_{h}\theta & = & b_{i}\delta x^{i}\wedge dt+\frac{1}{2}b_{ij}\delta x^{j}\wedge\delta x^{i},\nonumber \\
d_{\Phi}\theta & = & R_{i}^{j}a_{j}\delta x^{i}\wedge dt+\frac{1}{2}(g_{ik}R_{j}^{k}-g_{jk}R_{i}^{k})\delta x^{j}\wedge\delta x^{i},\label{eq:der}\\
\nabla d\theta & = & \nabla b_{i}\delta x^{i}\wedge dt+\nabla
a_{i}\delta y^{i}\wedge dt+\frac{1}{2}\nabla b_{ij}\delta
x^{j}\wedge\delta x^{i} +\nabla g_{ij}\delta y^{j}\wedge\delta
x^{i}.\nonumber \end{eqnarray*} In view of above formulae, one can
immediately see that Helmholtz-type conditions $(H_{1})-(H_{4})$
are equivalent with conditions \eqref{eq:43}-\eqref{eq:46}. First
three conditions \eqref{eq:43}-\eqref{eq:45} are usually known as
the classic Helmholtz conditions of the multiplier matrix
$g_{ij}=\partial \theta_i/\partial y^j$. Fourth classic Helmholtz
condition
$$ \frac{\partial g_{ij}}{\partial y^{k}}= \frac{\partial g_{ik}}{\partial
y^{j}}$$ is identically satisfied in view of last notation
\eqref{eq:40}.

The requirement that semi-basic $1$-form $\theta$ is not
$d_J$-closed implies that $i_Sd_J\theta\neq 0$. Therefore
$i_Ji_Sd\theta\neq 0$ and hence $i_Sd\theta \neq 0$, which assures
that a solution of equation \eqref{eq:49} is not trivial. Locally,
$d_J\theta \neq 0$ implies  $a_i$ from formula \eqref{eq:40} is
not identically zero, and hence $i_{S}d\theta=\nabla a_{i}\delta
x^{i}-a_{i}\delta y^{i}$ is a non trivial, dual symmetry of the
semispray $S$.

Theorem \ref{thm:LST1} corresponds to the time independent case
studied in Theorem 4.3, \cite{bucataru09}, where it is shown that
for $0$-homogeneous semi-basic $1$-forms only two of the Helmholtz
conditions are independent. These two conditions are $d_J\theta=0$
and $d_h\theta=0$ and appear in both Theorem \ref{thm:LST1} as
well as Theorem 4.3 from \cite{bucataru09}.

Theorem \ref{thm:LST2} corresponds to the time independent case
studied in Theorem 4.1, \cite{bucataru09}, where conditions
$(H_{1})-(H_{4})$ are equivalent with the fact that a semi-basic
$1$-form is $d_J$, $d_h$ and $d_{\Phi}$-closed and satisfies
$\nabla d\theta=0$. It is important to emphasize that in general,
for the time dependent case, Lagrangian semisprays are not
conservative. Theorem \ref{thm:LST2} refers to the class of
Lagrangian semisprays (conditions $(H_{1})-(H_{4})$ are satisfied)
that have symmetries as well (condition $(DS)$ is satisfied).

\begin{acknowledgement*}
The authors acknowledge fruitful comments from O.~Krupkov\'a
regarding some regularity aspects of the inverse problem and want
to express their thanks to W. Sarlet for many discussions that
clarified the main results of the paper.

The authors have been supported by CNCSIS grant IDEI 398 from the
Romanian Ministry of Education.
\end{acknowledgement*}

\end{document}